\documentclass[12pt]{amsart}
\usepackage{amssymb,latexsym}
\usepackage{amsfonts}
\usepackage{amsmath}
\usepackage[colorlinks,linkcolor=blue,anchorcolor=blue,citecolor=blue]{hyperref}
\usepackage{algorithm}
\usepackage{enumerate}
\usepackage{algpseudocode}
\usepackage{verbatim}
\usepackage{graphicx}

\newcommand\Q{{\mathbb Q}}

\newcommand\Z{{\mathbb Z}}

\newtheorem{theorem}{Theorem}[section]

\newtheorem{corollary}[theorem]{Corollary}
\newtheorem{proposition}[theorem]{Proposition}
\newtheorem{question}[theorem]{Question}
\theoremstyle{definition}

\newtheorem{example}[theorem]{Example}

\newtheorem{remark}[theorem]{Remark}

\numberwithin{equation}{section}

\begin{document}

\title[On the congruences of Eisenstein series]{On the congruences of Eisenstein series with polynomial indexes}

\author{Su Hu}
\address{Department of Mathematics, South China University of Technology, Guangzhou 510640, China}
\email{mahusu@scut.edu.cn}

\author{Min-Soo Kim}
\address{Department of Mathematics Education, Kyungnam University, Changwon, Gyeongnam 51767, Republic of Korea}
\email{mskim@kyungnam.ac.kr}

\author{Min Sha}
\address{School of Mathematical Sciences, South China Normal University, Guangzhou, 510631, China}
\email{shamin@scnu.edu.cn}



\begin{abstract}
In this paper, based on Serre's $p$-adic family of Eisenstein series, 
we prove a general family of congruences for Eisenstein series $G_k$ in the form
$$
\sum_{i=1}^n g_i(p)G_{f_i(p)}\equiv g_0(p)\mod p^N, 
$$
where   $f_1(t),\ldots,f_n(t)\in\Z[t]$ are non-constant integer polynomials with positive leading coefficients 
and  $g_0(t),\ldots,g_n(t)\in\Q(t)$ are rational functions. 
This generalizes the classical von Staudt--Clausen's and Kummer's congruences of Eisenstein series, 
and also yields some new congruences.  
\end{abstract}

\maketitle

\section{Introduction}
\label{Intro}

\subsection{Motivation}

Let $E_{k}$ ($k\geq 4$ even) be  normalized  Eisenstein series of weight $k$ for the modular group $\textrm{SL}_{2}(\mathbb{Z})$ given by the following $q$-expansion:
\begin{equation*}\label{Eisenstein}
E_{k}=1-\frac{2k}{B_{k}}\sum_{j=1}^{\infty}\sigma_{k-1}(j)q^{j},
\end{equation*}
where  $q=e^{2\pi i\tau}$, $B_{k}$ is the $k$-th Bernoulli number and $\sigma_{k-1}(j)=\sum_{d\mid j}d^{k-1}$.   
$E_{k}$ can be regarded as formal power series in the indeterminate $q$. 
If $f,g\in\mathbb{Q}[[q]]$ are power series and $N$ is a natural number, $f\equiv g\mod N$ means that $f$
and $g$ are both $N$-integral and the congruence holds coefficientwise \cite[page 132]{Gekeler}. 

In what follows, we assume that   $p$ is an odd prime. 

Several well-known congruences of $E_{k}$  have been given in \cite[page 164, Theorem 7.1]{Lang}. 
For example, from  von Staudt--Clausen's and Kummer's  congruences for Bernoulli numbers,
one can easily obtain (see \cite[Equation (1.3)]{Gekeler})
\begin{equation*}\label{E-Congruence}
\begin{aligned} 
&\quad\quad\quad E_{k}\equiv 1\mod p^{r} \quad \textrm{if} \quad k\equiv 0 \mod (p-1)p^{r-1} \\
&\textrm{and}\\
&\quad\quad\quad E_{k}\equiv E_{l} \mod p^{r} \quad \textrm{if} \quad k\equiv l \mod (p-1)p^{r},
\end{aligned}
\end{equation*}
for $k,l\geq r+1$, and $(k,p), (l,p)$ are regular. 
The last condition means that $p$ does not divide (the numerator of) $B_{k}$. 
Since this condition depends only on the residue class of $k \mod p-1$, it holds
simultaneously for $k$ and $l$ (see \cite[Equation (1.3)]{Gekeler}). By using Serre's theory of $p$-adic moular forms~\cite{Serre} and viewing the coefficients of Eisenstein series as Iwasawa functions \cite[Theorem 4.7]{Gekeler}, Gekeler \cite{Gekeler} proved several congruences of the shape $E_{k+l}\equiv E_{k}\cdot E_{l}$
modular prime power.

In this paper, we study congruence relations of Serre's normalized Eisenstein series (see \cite[page 194]{Serre})
\begin{equation}\label{eq:Gk}
G_{k}=-\frac{B_{k}}{2k}+\sum_{j =1}^{\infty}\sigma_{k-1}(j)q^{j}, \quad \textrm{$k \ge 4$ even.}
\end{equation} 
For further deductions, we make a convention that 
\begin{equation}    \label{G-convention}
G_k = 0, \quad \textrm{if $k \in \Z$ but $k$ is not even greater than $2$}. 
\end{equation} 

Let $f(t) \in \Z[t]$ have positive leading coefficient and satisfy $f(1)=0$. 
Then, von Staudt--Clausen's congruence of Bernoulli numbers in polynomial index (see \cite[Equation (1.2)]{Rosen}) implies that 
\begin{equation}\label{G-von}
2pf(p)G_{f(p)}\equiv 1 \mod p
\end{equation}
for every sufficiently large  prime $p$ (note that $f(p)$ is even because $f(1)=0$).  

Besides, let $f(t)$, $g(t)\in\Z[t]$ be distinct non-constant polynomials with positive leading coefficient, and suppose that $f(1)=g(1)\neq 0$. 
Let $d$ be the largest power of $t$ dividing $f(t)-g(t)$. 
Then, by using Kummer's congruence of Bernoulli numbers in polynomial index (see \cite[Equation (1.3)]{Rosen}), one can obtain 
\begin{equation}\label{G-Kummer}
G_{f(p)}\equiv G_{g(p)} \mod p^{d+1}
\end{equation}
for every sufficiently large odd prime $p$.

In order to generalize \eqref{G-von} and \eqref{G-Kummer}, we  consider the following problem:
\begin{question}
Given polynomials $f_1(t),\ldots,f_n(t)\in\Z[t]$ with positive leading coefficient, 
rational functions $g_0(t),g_1(t),\ldots,g_n(t)\in\Q(t)$, and a positive integer $N$, determine whether the congruence
\begin{equation*}
\sum_{i=1}^n g_i(p)G_{f_i(p)}\equiv g_0(p)\mod p^N
\end{equation*}
is true for any sufficiently large prime $p$.
\end{question}

This is inspired by a recent work of Julian Rosen \cite{Rosen}. He investigated a similar problem for Bernoulli numbers \cite[Question 1.1]{Rosen}, 
and he also obtained a very general criterion (see \cite[Theorem 1.2]{Rosen}).
The main tool  is a Taylor expansion for the Kubota-Leopoldt's $p$-adic zeta functions (see \cite[Proposition 2.1]{Rosen}).  
As pointed out in \cite[page 1896]{Rosen}, the well-known Kummer's and von Staudt--Clausen's  congruences of Bernoulli
numbers in polynomial index which have been given in \cite[Sections~9.5 and 11.4.2]{Cohen} can be deduced from this criterion.

\subsection{Main results}

From now on, let $N$ be a fixed positive integer. 
Let $f_1(t),\ldots,f_n(t)\in\Z[t]$ be non-constant integer polynomials with positive leading coefficients, 
and let $g_0(t),\ldots,g_n(t)\in\Q(t)$ be rational functions. 
Write $v_t$ for the $t$-adic valuation on $\Q(t)$, and set
$$
M=\min_{i=1,\ldots,n}\{v_t(g_i(t))\}.
$$
Here, we fix a convention that $v_t(0) = \infty$.  
For the $p$-adic valuation $v_p$, as usual we fix the convention $v_p(0) = \infty$. 

We define the following four conditions: 

{\bf C1}:
\begin{align*}
v_t\Big(g_{0}(t) 
& +\frac{1}{2}\big(1-\frac{1}{t}\big)\sum_{\substack{i=1\\f_{i}(1)=0}}^{n}g_{i}(t)f_{i}(t)^{-1} \\
& +\frac{1}{2}\sum_{\substack{i=1\\ \textrm{$f_{i}(1)\geq 4$ even}}}^{n}\frac{B_{f_{i}(1)}}{f_{i}(1)}(1-t^{f_{i}(1)-1})g_{i}(t)\Big)\geq N; 
\end{align*}

{\bf C2}: for every even integer $l \le 2$ and every $0 \leq m \leq N-M-1$,
$$
v_t\left(\sum_{\substack{i=1\\f_i(1)= l}}^{n} g_i(t)f_i(t)^{m}\right)\geq N-m; 
$$

{\bf C3}: for every even integer $l \ge 4$ and every $1\leq m\leq N-M-1$,
$$
v_t\left(\sum_{\substack{i=1\\f_i(1)= l}}^{n} g_i(t)(f_i(t)^{m}-l^{m})\right)\geq N-m; 
$$

{\bf C4}: for every even integer $l\ge 4$, 
$$
v_t\left(\sum_{\substack{i=1\\f_i(1)= l}}^{n} g_i(t)\right)\geq N.
$$

We remark that if $N-M < 1$, then the condition {\bf C2} automatically holds; 
and if $N-M \le 1$,  the condition {\bf C3} also holds automatically. 
Besides, if $f_i(1) \le 3$ for each $1 \le i \le n$, then the conditions 
{\bf C3} and {\bf C4} hold automatically.

In order to make our main result effective, let $P$ be a positive integer safisfying: 
\begin{itemize}
\item $P \ge  N -M +3$;

\item $P \ge |f_i(1)| + 1$ for each $1 \le i \le n$; 

\item for each $1 \le i \le n$ and any integer $j > P$, $f_i(j)> \max \{3,N\}$; 

\item for each $1 \le i \le n$, write $g_i(t) = t^{d_i}h_i(t)$ with $h_i(0) \ne 0$ for some integer $d_i \ge 0$, $P$ is not less than the numerator and denominator of $|h_i(0)| \in \Q$; (Under this condition, for any prime $p>P$ we have $v_p(g_i(p))=v_t(g_i(t))$ for each $1 \le i \le n$; see \cite[Proposition 3.1]{Rosen}.)

\item for each valuated function, say $h(t)$, in the $v_t$ valuation in the conditions {\bf C1}, {\bf C2}, {\bf C3} and {\bf C4} 
(for instance, in {\bf C4}, $h(t)=\sum_{\substack{i=1\\f_i(1)= l}}^{n} g_i(t)$), 
write $h(t) = t^{d}q(t)$ with $q(0) \ne 0$ for some integer $d \ge 0$, $P$ is not less than the numerator and denominator of $|q(0)| \in \Q$. (Under this condition, for any prime $p>P$ we have $v_p(h(p))=v_t(h(t))$ for each such function $h(t)$; see \cite[Proposition 3.1]{Rosen}.)
\end{itemize}

When the initial data $N, f_1, \ldots, f_n, g_0, \ldots, g_n$ are given, 
the verification of the conditions {\bf C1}, {\bf C2}, {\bf C3} and {\bf C4} is in fact a finite computation, 
and also it is easy to get an explicit choice for the integer $P$.

Our main result is the following congruence relation of Eisenstein series $G_{k}$. 

\begin{theorem}
\label{main}
The congruence
\begin{equation*}
\sum_{i=1}^n g_i(p)G_{f_i(p)}\equiv g_0(p)\mod p^N
\end{equation*}
holds for every odd prime $p > P$ if all the conditions {\bf C1}, {\bf C2}, {\bf C3} and {\bf C4} hold.  
\end{theorem}

Theorem~\ref{main}  is an analogue of \cite[Theorem~1.2]{Rosen} and moreover in an effective manner.  

The condition that all the polynomials $f_i$ are non-constant is for simplicity. 
But it is not essential, that is, for each $f_i$ that is constant, we can move the term 
$g_i(p)G_{f_i(p)}$ to the right hand side of the congruence in Theorem~\ref{main}. 

\begin{remark}
Using Theorem~\ref{main}, we can directly recover the congruences \eqref{G-von} and \eqref{G-Kummer}. 
For proving \eqref{G-von}, we choose $N=1$, $n=1$, $f_1(t)=f(t)$ satisfying $f(1)=0$, $g_0(t)=1$ and $g_1(t) = 2tf(t)$ in Theorem~\ref{main}; 
while for proving \eqref{G-Kummer}, we choose $n=2$, $f_{1}(t)=f(t)$, $f_{2}(t)=g(t)$, $g_{0}(t)=0$, $g_{1}(t)=1$, $g_{2}(t)=-1$ 
and $N=v_t(f-g)+1$  in Theorem~\ref{main} and notice the condition $f(1)=g(1)\neq 0$.  
\end{remark}

The following corollary is a generalization of \eqref{G-Kummer}. 

\begin{corollary}  \label{cor:Kummer}
In Theorem~\ref{main}, choose 
$$
N = 1 + \min_{1 \le i, j \le n} v_t(f_i - f_j),
$$ 
and assume that $f_1(1) = \ldots = f_n(1) \ne 0$, $g_0 = 0$, $g_1 + \ldots + g_n = 0$. 
Then, for any odd prime $p > P$, we have 
$$
\sum_{i=1}^n g_i(p)G_{f_i(p)} \equiv 0 \mod p^N. 
$$
\end{corollary}

The following corollary is a direct consequence of Theorem~\ref{main}. 

\begin{corollary}  \label{coro}
Assume that $f_i(1) \le 3$ for each $1 \le i \le n$. 
Then, The congruence
\begin{equation*}
\sum_{i=1}^n g_i(p)G_{f_i(p)}\equiv g_0(p)\mod p^N
\end{equation*}
holds for every odd prime $p > P$ if the following two conditions hold: 

$(1)$ 
\begin{align*}
v_t\Big(g_{0}(t)  +\frac{1}{2}\big(1-\frac{1}{t}\big)\sum_{\substack{i=1\\f_{i}(1)=0}}^{n}g_{i}(t)f_{i}(t)^{-1} \Big)\geq N; 
\end{align*}

$(2)$
 for every even integer $l \le 2$ and every $0 \leq m \leq N-M-1$,
$$
v_t\left(\sum_{\substack{i=1\\f_i(1)= l}}^{n} g_i(t)f_i(t)^{m}\right)\geq N-m. 
$$
\end{corollary}

From Corollary~\ref{coro}, one can get some more examples about congruence of Eisenstein series. 

\begin{example}
In Corollary~\ref{coro}, we choose $n=2$, $g_{0}(t)=0$, $g_{1}(t)=1$, $g_{2}(t)=-1$ 
and $N=v_t(f_1 - f_2) - 1$ such that  $f_1(0)f_2(0) \ne 0$, $f_1(1)=f_2(1)=0$ and $N \ge 1$, then we have 
 $$
 G_{f_1(p)}\equiv G_{f_2(p)} \mod p^{N}
 $$
 for any sufficiently large prime $p$. 
\end{example}

\begin{example}
In Corollary~\ref{coro}, we choose $n \ge 2$, $f_{i}(t)=a_i(t-1)$ for each $1 \le i \le n$, 
$g_{0}(t)=\frac{1}{2}(\frac{1}{a_2} + \ldots + \frac{1}{a_n}  - \frac{n-1}{a_1})$, $g_{1}(t)=(n-1)t$, $g_{i}(t)=-t$ for each $2 \le i \le n$,
and $N=2$, then 
we have for any prime $p > \max\{4, n-1\}$, 
$$
(n-1)pG_{a_1(p-1)} - \sum_{i=2}^{n}pG_{a_i(p-1)} \equiv \frac{1}{2}(\frac{1}{a_2} + \ldots + \frac{1}{a_n}  - \frac{n-1}{a_1}) \mod p^2.
$$
In particular, we have  for any  prime $p > 4$,
 $$
p G_{a_1(p-1)} - p G_{a_2(p-1)} \equiv  \frac{1}{2}(\frac{1}{a_2} - \frac{1}{a_1}) \mod p^2. 
 $$
\end{example}

For the proof  of Theorem~\ref{main},  the approach is similar as in \cite{Rosen}, 
but it indeed needs some extra considerations in the setting of Eisenstein series. 
For example, we need a new ingredient, that is, a Taylor expansion for the non-constant coefficients of $p$-adic Eisenstein series in Proposition \ref{lemma3}, 
which will play a key role in the proof.

Our paper will be organized as follows. In Section~\ref{sec:Serre}, we give a brief recall of Serre's $p$-adic family of Eisenstein series. 
In Section~\ref{sec:proof}, we prove Theorem~\ref{main} and Corollary~\ref{cor:Kummer}.

\section{$p$-adic Eisenstein series}
\label{sec:Serre}
In this section, we recall some facts about Serre's $p$-adic family of Eisenstein series \cite{Serre}; see also \cite{HK}. 
Recall that $p$ is an odd prime. 

Serre's normalized Eisenstein series has been defined in \eqref{eq:Gk}. 
We pass to the $p$-adic limit. 
Let $X=\mathbb{Z}_{p}\times\mathbb{Z}/(p-1)\mathbb{Z}$, where $\Z_p$ is the ring of $p$-adic integer. 
The integers $\Z$ are embedded into $X$ naturally by $j \mapsto (j,j)$. 
For $k\in X$ and $j \geq 1$, define 
$$
\sigma_{k-1}^{*}(j)=\sum_{\substack{d|j \\
(p,d)=1}}d^{k-1}
$$ 
(see \cite[page 205]{Serre}, and see \cite[page 201]{Serre} for the definition of $d^{k-1}$). 
If $k$ is even (that is, $k \in 2X$), there exists a sequence of even
integers $\{k_{i}\}_{i=1}^{\infty}$ such that $|k_{i}|\rightarrow\infty$ and $k_{i} \rightarrow k$ when $i\to\infty$. 
Then, the sequence $G_{k_{i}}=-\frac{B_{k_{i}}}{2k_{i}}+\sum_{j=1}^{\infty}\sigma_{k_{i}-1}(j)q^{j}$ has a limit: 
(see \cite[page 206]{Serre}) 
\begin{equation}\label{eq:Gk*}
G^*_{k}=a_{0}(G^*_{k})+\sum_{j=1}^\infty a_{j}(G^*_{k})q^{j}
\end{equation}
with $a_{0}(G^*_{k})=\frac{1}{2}\zeta^{*}(1-k)$ by defining $\zeta^{*}(1-k)=\lim_{i\to\infty}\zeta(1-k_{i}) $  
and $a_{j}(G^*_{k})=\sigma_{k-1}^{*}(j)$, where $\zeta(s)$ is the Riemann zeta function.  
The function $\zeta^{*}$ is thus defined on the odd elements of $X\setminus\{1\}$. 

Let $\chi$ be a Dirichlet character on $\mathbb{Z}_{p}$, and let $L_{p}(s,\chi)$ be the $p$-adic $L$-function.
We have the following result on $\zeta^{*}$.

\begin{theorem}[{see \cite[page 206, Th{\'e}or{\`e}me 3]{Serre}}] \label{serre} 
 If $(s,u)\not=1$ is an odd element of
$X=\mathbb{Z}_{p}\times\mathbb{Z}/(p-1)\mathbb{Z}$,  then
$$
\zeta^{*}(s,u)=L_{p}(s,\omega^{1-u}),
$$ 
where $\omega$ is the Teichm\"uller character.
\end{theorem}

For $k=(s,u)\in X$ and $u$ is even, by Theorem~\ref{serre} the coefficients of
$G_{k}^{*}=G_{s,u}^{*}$ are given by (see \cite[page 245]{Serre})
\begin{equation}\label{eq1}
\begin{aligned}
a_{0}(G_{s,u}^{*})&=\frac{1}{2}\zeta^{*}(1-s,1-u)=\frac{1}{2}L_{p}(1-s,\omega^{u}),\\
a_{j}(G_{s,u}^{*})&=\sum_{\substack{d|j \\
(p,d)=1}} d^{-1}\omega(d)^{u}\langle d\rangle^{s}, 
\end{aligned}
\end{equation}
where $\langle d \rangle = d / \omega(d) \equiv 1 \mod p$. 

Thus, the assignment
$$(s,u)\mapsto G_{s,u}^{*}$$
gives a family of $p$-adic modular forms parametrized by the group of weights $X$.

For any even integer $k\geq 4$, we first write 
\begin{equation}   \label{eq:Gka}
G_{k}=a_{0}(G_{k})+\sum_{j=1}^\infty a_{j}(G_{k})q^{j},  
\end{equation} 
where $a_{0}(G_{k})=-\frac{B_k}{2k}, \,  a_{j}(G_{k}) = \sigma_{k-1}(j)$; 
and then from \eqref{eq1},  we  have
\begin{equation}\label{eq2}
\begin{aligned}
a_{0}(G_{k}^{*})&=a_{0}(G_{k,k}^{*})=\frac{1}{2}\zeta^{*}(1-k,1-k)   \\
&=\frac{1}{2}L_{p}(1-k,\omega^{k}) =-\frac{1-p^{k-1}}{2}\frac{B_{k}}{k} \\
&=(1-p^{k-1})a_{0}(G_{k}),
\end{aligned}
\end{equation}
where we also use the relation between $L_p$ and Bernoulli numbers 
(see, for instance, the first paragraph in the proof of \cite[Proposition 2.1]{Rosen}), 
and
\begin{equation}\label{eq2-1}
a_{j}(G_{k}^{*})=a_{j}(G_{k,k}^{*})
=\sum_{\substack{d|j \\(p,d)=1}} d^{-1}\omega(d)^{k}\langle d\rangle^{k}    
=\sum_{\substack{d|j \\(p,d)=1}}d^{k-1}.
\end{equation}

The proof of our main result is based on the following relationship between the $p$-adic Eisenstein Series $G_{k}^{*}$ and the Eisenstein series $G_{k}$:
\begin{equation}\label{con} 
G_{k}\equiv G_{k}^{*}\mod p^{k-1},  \qquad \textrm{$k \ge 4$ even}, 
\end{equation}
which can be easily deduced from \eqref{eq:Gk*},  \eqref{eq:Gka}, \eqref{eq2} and \eqref{eq2-1}.

As in \eqref{G-convention}, we also make a convention that 
\begin{equation}    \label{G*-convention}
G_k^* = 0, \quad \textrm{if $k \in \Z$ but $k$ is not even greater than $2$}. 
\end{equation}

\section{Proofs of the main results}
\label{sec:proof}

Recall that $p$ is an odd prime. 
For the proof, we need some preparations. The first one follows from \eqref{eq2} and \cite[Proposition 2.1]{Rosen}  directly.

 \begin{proposition}\label{lemma1}
Let  $l$ be an even residue class modulo $p-1$. 
Then, there exist coefficients $a_m^{(0)}(p,l)\in\Q_p, m=0,1,2,\ldots$, such that for every even integer  
 $k \ge 4$ with $k\equiv l \pmod{p-1}$, there is a convergent $p$-adic series identity
\begin{equation}
\label{B}
a_{0}(G_{k}^{*})=-\frac{1-p^{k-1}}{2}\frac{B_{k}}{k}=-\frac{1}{2}\sum_{m= 0}^{\infty}a_{m}^{(0)}(p,l)k^{m-1}.
\end{equation}
The coefficients $a_{m}^{(0)}(p,l)$ satisfy the following conditions:
\begin{enumerate}
\item[{\rm(1)}]
\[
a_{0}^{(0)}(p,l)=\begin{cases}1-\frac{1}{p}&\text{ if }l \equiv 0 \mod p-1,\\
0&\text{ otherwise,}\end{cases}
\]
\item[{\rm(2)}] for all $m$, $p$ and $l$,
\[
v_p(a_{m}^{(0)}(p,l))\geq \frac{p-2}{p-1}m - 2,
\]
\item[{\rm(3)}] for $p\geq m+2$ and all $l$,
\[
v_p(a_{m}^{(0)}(p,l))\geq m-1.
\]
\end{enumerate}
\end{proposition}

Using Proposition~\ref{lemma1}, we obtain a congruence relation for the coefficient $a_0(G^*_k)$ 
in polynomial index. 

\begin{proposition}
\label{lemma2}
The congruence
\begin{equation*}
\sum_{i=1}^n g_i(p)a_{0}(G_{f_i(p)}^{*})\equiv g_0(p)\mod p^N
\end{equation*}
holds for every prime $p > P$ if the conditions {\bf C1},  {\bf C2} and {\bf C3} hold. 
\end{proposition}

\begin{proof} 
We extend the proof of \cite[Theorem 1.2]{Rosen} to our case.

Since $p > P$ and noticing the choice of $P$, we know that $f_i(p) \ge 4$ for each $1 \le i \le n$.  
In view of the convention \eqref{G*-convention}, we  consider the quantity
\begin{equation*}
A^{(0)}(p)=g_{0}(p)-\sum_{i=1}^{n}g_{i}(p)a_{0}(G_{f_{i}(p)}^{*}).
\end{equation*}
By Proposition~\ref{lemma1},  we have
\begin{equation*}
\begin{aligned}
A^{(0)}(p)
&=g_{0}(p)+\sum_{\substack{i=1 \\ \textrm{$f_i(p)$ even}}}^{n}g_{i}(p)\left(\frac{1}{2}\sum_{m=0}^{\infty}a_{m}^{(0)}(p,f_{i}(p))f_{i}(p)^{m-1}\right)\\
&=g_{0}(p)+\sum_{\substack{h\in\mathbb{Z}/(p-1)\mathbb{Z}\\ \textrm{$h$ even, $m\geq 0$}}} \sum_{\substack{i=1\\f_{i}(p)\equiv h \!\!\!\!\mod p-1}}^{n}\frac{1}{2}g_{i}(p)f_{i}(p)^{m-1} a_{m}^{(0)}(p,h).
\end{aligned}
\end{equation*}
Since $f_{i}(p)\equiv f_{i}(1) \pmod{p-1}$ for each $1 \le i \le n$,  we have
\begin{equation*}
\begin{aligned}
A^{(0)}(p)&=g_{0}(p)+\sum_{\substack{ \textrm{ even $l\in\mathbb{Z}$} \\ m\geq 0}}\sum_{\substack{i=1\\f_{i}(1)=l}}^{n}\frac{1}{2}g_{i}(p)f_{i}(p)^{m-1} a_{m}^{(0)}(p,l)\\
&=g_{0}(p)+\frac{1}{2}\sum_{\substack{l\leq 2~~ \textrm{even}\\ l \ne 0,\,  m\geq 0}}\sum_{\substack{i=1\\f_{i}(1)=l}}^{n}g_{i}(p)f_{i}(p)^{m-1}a_{m}^{(0)}(p,l)
\\&\quad+\frac{1}{2}\sum_{m\geq 0}\sum_{\substack{i=1\\f_{i}(1)=0}}^{n}g_{i}(p)f_{i}(p)^{m-1}a_{m}^{(0)}(p,0)
 \\&\quad+\frac{1}{2}\sum_{\substack{l\geq 4~~ \textrm{even}\\ m\geq 0}}\sum_{\substack{i=1\\f_{i}(1)=l}}^{n}g_{i}(p)f_{i}(p)^{m-1}a_{m}^{(0)}(p,l),
\end{aligned}
\end{equation*}
which, by Proposition~\ref{lemma1}~(1), becomes 
\begin{equation}\label{eq9}
\begin{aligned}
& A^{(0)}(p)=g_{0}(p)+\frac{1}{2}\sum_{\substack{l\leq 2~~ \textrm{even}\\ l \ne 0, \, m\geq 0}}\sum_{\substack{i=1\\f_{i}(1)=l}}^{n}g_{i}(p)f_{i}(p)^{m-1}a_{m}^{(0)}(p,l) \\
&\,\, +\frac{1}{2}\left(1-\frac{1}{p}\right)\sum_{\substack{i=1\\f_{i}(1)=0}}^{n}g_{i}(p)f_{i}(p)^{-1}
+\frac{1}{2}\sum_{m\geq 1}\sum_{\substack{i=1\\f_{i}(1)=0}}^{n}g_{i}(p)f_{i}(p)^{m-1}a_{m}^{(0)}(p,0) \\
 &\,\, +\frac{1}{2}\sum_{\substack{l\geq 4~~ \textrm{even}\\ m\geq 0}}\sum_{\substack{i=1\\f_{i}(1)=l}}^{n}g_{i}(p)f_{i}(p)^{m-1}a_{m}^{(0)}(p,l).
\end{aligned}\end{equation}

Due to the choice of $P$ and $p>P$, we have $p > |f_i(1)| +1$ for each $1\le i \le n$. 
So, for any even $l$ satisfying $l=f_i(1)$ for some $1 \le i \le n$,  it can not happen that $l \equiv 0 \pmod{p-1}$, 
which together with Proposition~\ref{lemma1}~(1)  implies that 
\begin{equation}   \label{eq:a00}
a_{0}^{(0)}(p,l)=0.
\end{equation} 
Thus, from \eqref{eq2}, \eqref{B} and \eqref{eq:a00}, for any even $l \ge 4$ satisfying $l=f_i(1)$ for some $1 \le  i \le n$,  we have
\begin{equation}
\label{C}
\begin{split}
a_{1}^{(0)}(p,l) & =-2a_{0}(G_{l}^{*})-\sum_{m\geq 2}a_{m}^{(0)}(p,l)l^{m-1}  \\
& = (1-p^{l-1}) \frac{B_l}{l} -\sum_{m\geq 2}a_{m}^{(0)}(p,l)l^{m-1}. 
\end{split}
\end{equation}
Substituting \eqref{eq:a00} and \eqref{C} into \eqref{eq9}, we have
\begin{equation}\label{eqnew2}
\begin{aligned}
& A^{(0)}(p)=g_{0}(p)+\frac{1}{2}\sum_{\substack{l\leq 2~~ \textrm{even}\\ l \ne 0, \, m\geq 1}}\sum_{\substack{i=1\\f_{i}(1)=l}}^{n}g_{i}(p)f_{i}(p)^{m-1}a_{m}^{(0)}(p,l)  \\
&\,\, +\frac{1}{2}\left(1-\frac{1}{p}\right)\sum_{\substack{i=1\\f_{i}(1)=0}}^{n}g_{i}(p)f_{i}(p)^{-1}
+\frac{1}{2}\sum_{m\geq 1}\sum_{\substack{i=1\\f_{i}(1)=0}}^{n}g_{i}(p)f_{i}(p)^{m-1}a_{m}^{(0)}(p,0)  \\
&\,\,  +\frac{1}{2}\sum_{l\geq 4~~\textrm{even}}\sum_{\substack{i=1\\f_{i}(1)=l}}^{n}\frac{B_{l}}{l}(1-p^{l-1})g_{i}(p)  \\
&\,\,  +\frac{1}{2}\sum_{\substack{l\geq 4~~ \textrm{even}\\ m\geq 2}}\sum_{\substack{i=1\\f_{i}(1)=l}}^{n}g_{i}(p)(f_{i}(p)^{m-1}-l^{m-1})a_{m}^{(0)}(p,l).\end{aligned}
\end{equation}

Under the condition {\bf C1} and noticing the choices of $p$ and $P$, we have 
\begin{equation}\label{eq:C1}
\begin{aligned}
g_{0}(p) & +\frac{1}{2}\left(1-\frac{1}{p}\right)\sum_{\substack{i=1\\f_{i}(1)=0}}^{n}g_{i}(p)f_{i}(p)^{-1}
 \\& + \frac{1}{2}\sum_{l\geq 4~~\textrm{even}}\sum_{\substack{i=1\\f_{i}(1)=l}}^{n} \frac{B_{l}}{l}(1-p^{l-1})g_{i}(p)\equiv 0\mod p^{N}.
 \end{aligned}
 \end{equation}
 
For every even integer $l \le 2$  satisfying $l = f_i(1)$ for some $1 \le i \le n$, 
under the condition {\bf C2} and due to the choices of $p$ and $P$, for any $1 \le m \le N-M$ we have 
 $$
 v_p\Big(\sum_{\substack{i=1\\f_{i}(1)=l}}^{n} g_{i}(p)f_{i}(p)^{m-1}\Big) 
 = v_p\Big(\sum_{\substack{i=1\\f_{i}(1)=l}}^{n} g_{i}(t)f_{i}(t)^{m-1}\Big)
 \ge N - (m-1); 
 $$
and by Proposition~\ref{lemma1}~(3)   and noticing $p \ge N-M+2 \ge m + 2$ due to the choice of $P$, we have 
$$
v_p(a_{m}^{(0)}(p,l)) \ge m - 1; 
$$
and so we obtain 
\begin{equation}
\label{eq:C21}
 v_p\Big(\sum_{\substack{i=1\\f_{i}(1)=l}}^{n} g_{i}(p)f_{i}(p)^{m-1}a_{m}^{(0)}(p,l)\Big) \ge N, \quad 1 \le m \le N-M. 
\end{equation}
If $m \ge N-M+1$ and $p \ge m+2$, then by Proposition~\ref{lemma1}~(3), we have 
$v_p(a_{m}^{(0)}(p,l)) \ge m - 1$, which together with $v_p(g_i(p)) = v_t(g_i(t)) \ge M$ for each $1 \le i \le n$ (due to the choices of $p$ and $P$) 
implies that for $m \ge N-M+1$ and $p \ge m+2$, for some $j$ with $f_j(1) = l$, 
\begin{equation}
\label{eq:C22}
\begin{split}
 v_p\Big(\sum_{\substack{i=1\\f_{i}(1)=l}}^{n} g_{i}(p)f_{i}(p)^{m-1}a_{m}^{(0)}(p,l)\Big) 
 & \ge v_p(g_j(p)) + v_p(a_{m}^{(0)}(p,l)) \\
 & \ge M + m -1 \ge  N. 
\end{split}
\end{equation}
If $m \ge N-M+1$ and $p \le m+1$, then by Proposition~\ref{lemma1}~(2) and noticing $v_p(g_i(p)) \ge M$ and $p >P \ge N-M+3$, we obtain 
\begin{equation}
\label{eq:C23}
\begin{split}
 v_p\Big(& \sum_{\substack{i=1\\f_{i}(1)=l}}^{n} g_{i}(p)f_{i}(p)^{m-1}a_{m}^{(0)}(p,l)\Big) 
  \ge M + v_p(a_{m}^{(0)}(p,l)) \\
 & \ge M+ \frac{p-2}{p-1}m - 2 \ge M+p -4 \ge N. 
\end{split}
\end{equation}
Thus, under the condition {\bf C2} and combining \eqref{eq:C21}, \eqref{eq:C22} with \eqref{eq:C23},  
for every even integer $l\leq 2$ satisfying $l = f_i(1)$ for some $1 \le i \le n$ and any $m \ge 1$, we have  
\begin{equation}\label{eq:C2} 
\sum_{\substack{i=1\\f_{i}(1)=l}}^{n} g_{i}(p)f_{i}(p)^{m-1}a_{m}^{(0)}(p,l) \equiv 0\mod p^{N}.
\end{equation}

As the above, under the condition {\bf C3} and noticing the choices of $p$ and $P$, 
for every even $l\geq 4$ and $m\geq 2$, we have
\begin{equation}\label{eq:C3}
\sum_{\substack{i=1\\f_{i}(1)=l}}^{n} g_{i}(p)(f_{i}(p)^{m-1}-l^{m-1})a_{m}^{(0)}(p,l)  \equiv 0  \mod p^{N}. 
\end{equation}

Finally, by \eqref{eqnew2}, \eqref{eq:C1}, \eqref{eq:C2} and \eqref{eq:C3} we conclude that $A^{(0)}(p)\equiv 0\mod p^{N}$ for any prime $p>P$. 
This completes the proof. 
\end{proof}

As an analogue of Proposition~\ref{lemma1}, we  obtain a 
convergent $p$-adic series identity for each coefficient $a_{j}(G_{k}^{*}), j \ge 1$. 
The approach here is different from the one in \cite[Proposition 2.1]{Rosen}. 

\begin{proposition}\label{lemma3}
Let $p$ be an odd prime, $l$ an even residue class modulo $p-1$, and $j$ a positive integer. 
Then, there exist coefficients $a_m^{(j)}(p,l)\in\Q_p, m=0,1,2,\ldots$, 
such that for every even integer  $k \ge 4$ with $k\equiv l \pmod{p-1}$, 
there is a convergent $p$-adic series identity
\begin{equation*}
\label{D}
a_{j}(G_{k}^{*})=\sum_{m= 0}^{\infty} a_{m}^{(j)}(p,l)k^{m}.
\end{equation*}
The coefficients $a_m^{(j)}(p,l)$ satisfy the following conditions:
\begin{enumerate}
\item[{\rm(1)}] for all $m$, $p$, $l$,
\[
v_p(a_{m}^{(j)}(p,l))\geq \frac{p-2}{p-1} m,
\]
\item[{\rm(2)}] for $p\geq m+2$ and all $l$,
\[
v_p(a_{m}^{(j)}(p,l))\geq m.
\]
\end{enumerate}
\end{proposition}

\begin{proof} 
For $(s,u)\in X$ and $u$ is even, 
by \eqref{eq1}, we have 
\begin{equation}\label{eq5} 
a_{j}(G_{s,u}^{*})=\sum_{\substack{d\mid j \\(p,d)=1}}d^{-1}\omega(d)^{u}\langle d \rangle^{s}.
\end{equation}
Write $\langle d \rangle=1+pq_{d}$ with $q_{d}\in\mathbb{Z}_{p}$, we have 
$$
\langle d \rangle^{s}=\sum_{m=0}^{\infty} \binom{s}{m} p^{m}q_{d}^{m}.
$$
Substituting the above  into \eqref{eq5}, we have
\begin{equation*}
\begin{aligned}
a_{j}(G_{s,u}^{*})&=\sum_{\substack{d\mid j \\(p,d)=1}}d^{-1}\omega(d)^{u}\sum_{m=0}^{\infty} \binom{s}{m} p^{m}q_{d}^{m}\\
&=\sum_{m=0}^{\infty} \binom{s}{m}p^{m}\sum_{\substack{d\mid j\\(p,d)=1}}q_{d}^{m}d^{-1}\omega(d)^{u}. 
\end{aligned}
\end{equation*}
Thus, we obtain 
\begin{equation}
\label{eq:anG}
\begin{split}
a_{j}(G_{k}^{*})=a_{j}(G_{k,k}^{*})
& =\sum_{m=0}^{\infty} \binom{k}{m} p^{m} \sum_{\substack{d\mid j\\(p,d)=1}}q_{d}^{m}d^{-1}\omega(d)^{k} \\ 
& =\sum_{m=0}^{\infty} \binom{k}{m} p^{m} \sum_{\substack{d\mid j\\(p,d)=1}}q_{d}^{m}d^{-1}\omega(d)^{l}, 
\end{split}
\end{equation}
where the last equality comes from the fact that  $\omega(a)^{k}=\omega(a)^{l}$ if  $k\equiv l \pmod{p-1}$. 

For each $1\le m \le k$, we have 
\begin{equation}
\label{eq:binom}
\begin{split}
\binom{k}{m} & = \frac{k(k-1)\cdots (k-m+1)}{m!}   \\
& = \frac{1}{m!} (k^m + b_{m,m-1}k^{m-1} + \cdots + b_{m,1} k)
\end{split}
\end{equation}
for some integers $b_{m,1}, \ldots, b_{m,m-1} \in \Z$ dependings only on $m$. 

Substituting \eqref{eq:binom} into \eqref{eq:anG}, 
we obtain 
$$
a_{j}(G_{k}^{*})=\sum_{m=0}^{\infty}a_{m}^{(j)}(p,l)k^{m}
$$
for some $a_{m}^{(j)}(p,l)\in \mathbb{Q}_{p}$ satisfying
\begin{align*}
v_{p}(a_{m}^{(j)}(p,l)) 
& \geq  \min\{v_p(p^m/m!), v_p(p^{m+1}/(m+1)!), \ldots\} \\ 
& \ge m-\frac{m}{p-1}, 
\end{align*}
where the number of terms in the $\min$ function is finite and the last inequality follows from the fact that $v_p (m!) \le m/(p-1)$.   
This gives the conclusion (1) of the proposition. 
The conclusion (2) (in the case $p \geq m+2$) follows from (1) directly by noticing $v_{p}(a_{m}^{(j)}(p,l)) \in \Z$.   
\end{proof}

Applying Proposition~\ref{lemma3}, we can also obtain a congruence relation for the coefficient $a_j(G^*_k)$ 
in polynomial index. 

\begin{proposition}\label{lemma4}
For any integer $j \ge 1$, the congruence
\begin{equation*}
\sum_{i=1}^n g_i(p)a_{j}(G^*_{f_i(p)})\equiv 0\mod p^N
\end{equation*}
holds for every prime $p>P$ if the conditions {\bf C2}, {\bf C3} and {\bf C4} hold. 
\end{proposition}

\begin{proof}
We apply the same strategy as in the proof of Proposition~\ref{lemma2}. 
 
Since $p > P$ and noticing the choice of $P$, we know that $f_i(p) \ge 4$ for each $1 \le i \le n$.  
In view of the convention \eqref{G*-convention}, we  consider the quantity
\begin{equation*}
A^{(j)}(p)=\sum_{i=1}^{n}g_{i}(p)a_{j}(G^*_{f_{i}(p)}).
\end{equation*}
By Proposition~\ref{lemma3}, we have
\begin{equation*}
\begin{aligned}
A^{(j)}(p)
&=\sum_{\substack{i=1\\ \textrm{$f_i(p)$ even}}}^{n}g_{i}(p)\sum_{m=0}^{\infty}a_{m}^{(j)}(p,f_{i}(p))f_{i}(p)^{m}    \\
&=\sum_{\substack{h\in\mathbb{Z}/(p-1)\mathbb{Z}\\ \textrm{$h$ even, $m\geq 0$}}} \sum_{\substack{i=1\\f_{i}(p)\equiv h\!\!\!\!\mod p-1}}^{n} g_{i}(p)f_{i}(p)^{m}a_{m}^{(j)}(p,h).
\end{aligned}
\end{equation*}
Since $f_{i}(p)\equiv f_{i}(1)\mod p-1$, we have
\begin{equation}\label{eq9*}
\begin{aligned}
A^{(j)}(p)&=\sum_{\substack{\textrm{even $l\in\mathbb{Z}$} \\ m\geq 0}} \sum_{\substack{i=1\\f_{i}(1)=l}}^{n} g_{i}(p)f_{i}(p)^{m} a_{m}^{(j)}(p,l)\\
&=\sum_{\substack{l\leq 2~~ \textrm{even}\\ m\geq 0}}\sum_{\substack{i=1\\f_{i}(1)=l}}^{n} g_{i}(p)f_{i}(p)^{m}a_{m}^{(j)}(p,l)  \\
&\quad+\sum_{\substack{l\geq 4~~ \textrm{even}\\ m\geq 0}}\sum_{\substack{i=1\\f_{i}(1)=l}}^{n}g_{i}(p)f_{i}(p)^{m}a_{m}^{(j)}(p,l).
\end{aligned}
\end{equation}

For any even integer $l \ge 4$ satisfying $l=f_i(1)$ for some $1 \le i \le n$,  by Proposition~\ref{D} we have
\begin{equation*}
a_{0}^{(j)}(p,l)=a_{j}(G_{l}^{*})-\sum_{m=1}^{\infty}a_{m}^{(j)}(p,l)l^{m}.
\end{equation*}
Substituting the above equation into \eqref{eq9*}, we have
\begin{equation}\label{eqnew3}
\begin{aligned}
A^{(j)}(p)&=\sum_{\substack{l\leq 2~~ \textrm{even}\\ m\geq 0}}\sum_{\substack{i=1\\f_{i}(1)=l}}^{n}g_{i}(p)f_{i}(p)^{m}a_{m}^{(j)}(p,l)\\
  &\quad+\sum_{l\geq 4~~ \textrm{even}}a_{j}(G_{l}^{*}) \sum_{\substack{i=1\\f_{i}(1)=l}}^{n} g_{i}(p)
   \\&\quad+\sum_{\substack{l\geq 4~~ \textrm{even}\\ m\geq 1}}\sum_{\substack{i=1\\f_{i}(1)=l}}^{n} g_{i}(p)(f_{i}(p)^{m}-l^{m})a_{m}^{(j)}(p,l).
\end{aligned}
\end{equation}

As in the proof of Proposition~\ref{lemma2}, under the condition {\bf C2} and the choices of $p$ and $P$ and using Proposition~\ref{lemma3},  
for every even integer $l\leq 2$ and $m\geq 0$, we obtain 
\begin{equation}\label{eq:C2-2}
\sum_{\substack{i=1\\f_{i}(1)=l}}^{n}g_{i}(p)f_{i}(p)^{m}a_{m}^{(j)}(p,l) \equiv 0    \mod p^{N}.
 \end{equation}
Similarly, under the condition {\bf C3}, for every even integer $l\geq 4$ and $m\geq 1$, we have
\begin{equation}\label{eq:C3-2}
\sum_{\substack{i=1\\f_{i}(1)=l}}^{n} g_{i}(p)(f_{i}(p)^{m}-l^{m})a_{m}^{(j)}(p,l)\equiv 0\mod p^{N}. 
\end{equation}
Also, under the condition {\bf C4} and noticing $a_{j}(G_{l}^{*}) \in \Z_p$ by \eqref{eq2-1},  
for every even integer $l\geq 4$ we have  
\begin{equation}\label{eq:C4} 
a_{j}(G_{l}^{*}) \sum_{\substack{i=1\\f_{i}(1)=l}}^{n} g_{i}(p)\equiv 0\mod p^{N}.
\end{equation}

Finally, by \eqref{eqnew3}, \eqref{eq:C2-2}, \eqref{eq:C3-2} and \eqref{eq:C4}, 
 we conclude that $A^{(j)}(p)\equiv 0\mod p^{N}$ for any prime $p>P$. 
 This completes the proof. 
\end{proof}

We are now at the point to prove Theorem \ref{main}.

\begin{proof}[Proof of Theorem \ref{main}]
Since $p>P$ and noticing the choice of $P$,  we have $f_{i}(p) > N$ for each $1 \le i \le n$.
Thus, by \eqref{con}, for any $1 \le i \le n$ with even $f_i(p)$ we have
\begin{equation}   \label{eq:GG1}
G_{f_{i}(p)}\equiv G_{f_{i}(p)}^{*}\mod p^{N}.
\end{equation}
Otherwise if $f_i(p)$ is odd, then by the conventions \eqref{G-convention} and \eqref{G*-convention}, 
we have $G_{f_{i}(p)} = G^*_{f_{i}(p)}=0$, and so \eqref{eq:GG1} still holds. 
On the other hand, by Propositions~\ref{lemma2} and \ref{lemma4}, we directly obtain 
\begin{equation}  \label{eq:GG2}
\sum_{i=1}^n g_i(p)G^*_{f_i(p)} \equiv g_0(p)\mod p^N
\end{equation}
 for every prime $p > P$ if all the conditions {\bf C1}, {\bf C2}, {\bf C3} and {\bf C4}  hold.  
The desired result  now follows from \eqref{eq:GG1} and \eqref{eq:GG2}.
\end{proof}

Finally, we prove Corollary~\ref{cor:Kummer}. 

\begin{proof}[Proof of Corollary~\ref{cor:Kummer}]
First, by assumption, it is easy to see that the conditions {\bf C1} and {\bf C4} hold. 

Since $f_1(1) = \ldots = f_n(1)$ and $g_1 + \ldots + g_n = 0$, for verifying the conditions {\bf C2} and {\bf C3}, 
it suffices to show that for any $m \ge 1$, 
$$
v_t\Big( \sum_{i=1}^{n} g_i(t)f_i(t)^m\Big) \ge N-1. 
$$

Now, we first prove the case $m=1$. Let $d = \min_{1 \le i,j \le n} v_t(f_i - f_j)$. Then, $N = d + 1$. 
The case $d=0$ is trivial. Assume $d \ge 1$ and write $F(t) = \sum_{i=1}^{n} g_i(t)f_i(t)$. 
To prove $v_t(F) \ge d$, it suffices to show that $F^{(k)}(0) = 0$ for any $0 \le k \le d - 1$, where $F^{(k)}$ denotes the $k$-th derivative of $F$. 
Note that 
$$
F^{(k)}(t) = \sum_{j=0}^{k} \binom{k}{j} \sum_{i=1}^{n} g_i^{(j)}(t) f_i^{(k-j)}(t).
$$
Since $0 \le k \le d - 1$, by definition we have $f_1^{(k-j)}(0) = \ldots = f_n^{(k-j)}(0)$, and so 
$$
\sum_{i=1}^{n} g_i^{(j)}(0) f_i^{(k-j)}(0) =f_1^{(k-j)}(0) \sum_{i=1}^{n} g_i^{(j)}(0)  = 0, 
$$
where we use the assumption $g_1 + \ldots + g_n = 0$. 
Hence, we have $F^{(k)}(0) = 0$ for any $0 \le k \le d - 1$. 
This completes the proof of the case $m=1$. 

For $m \ge 2$, we have 
$$
\min_{1 \le i,j \le n} v_t(f_i^m - f_j^m) \ge \min_{1 \le i,j \le n} v_t(f_i - f_j) =d . 
$$
Hence, applying the same argument as the above, we obtain 
$$
v_t\Big( \sum_{i=1}^{n} g_i(t)f_i(t)^m\Big) \ge N-1. 
$$
The desired result now follows. 
\end{proof}

 \section*{Acknowledgement}
The authors thank Julian Rosen for his detailed explanation of an equality in his paper. 
Su Hu is supported by the Natural Science Foundation of Guangdong Province, China (No. 2020A1515010170).  Min-Soo Kim is supported by the National Research Foundation of Korea (NRF) grant funded by the Korea government (MSIT) (No. 2019R1F1A1062499). Min Sha was partly supported by a Macquarie University Research Fellowship.

\bibliography{central}

\end{document}